\documentclass[12pt,draftcls,onecolumn]{IEEEtran}

\IEEEoverridecommandlockouts                              

\makeatletter
\let\NAT@parse\undefined
\makeatother

\usepackage{graphicx}
\usepackage{bm}
\usepackage{epsfig}
\usepackage{amsthm}
\usepackage{amssymb,amsfonts,amsmath}
\usepackage{subfigure}
\usepackage{setspace}
\usepackage{enumerate}
\usepackage[numbers]{natbib}
\usepackage{color}
\usepackage{dsfont}
\usepackage{verbatim}

\newtheorem{thm}{Theorem}
\newtheorem{lem}{Lemma}
\newtheorem{rem}{Remark}
\newtheorem{definition}{Definition}
\newtheorem{prop}{Proposition}
\newtheorem{ex}{Example}

\newtheorem{cor}{Corollary}

\def \b {\beta}

\def \< {\langle}
\def \> {\rangle}

\def \O {\Omega}

\def \~{\tilde}

\def \b {\beta}
\def \O {\Omega}
\def \L {\mathcal{L}}
\def \M {\mathcal{M}}

\begin{document}
\title{Moment-Based Ensemble Control}

\author{Vignesh~Narayanan,~\IEEEmembership{Member,~IEEE,}
        Wei~Zhang,~\IEEEmembership{Member,~IEEE,}
        and~Jr-Shin~Li,~\IEEEmembership{Senior~Member,~IEEE}
\thanks{*This work was supported in part by the National Science Foundation under the awards CMMI-1462796 and ECCS-1509342, and by the Air Force Office of Scientific Research under the award FA9550-17-1-0166.}
\thanks{V. Narayanan, W. Zhang, and J.-S. Li are with the Department
of Electrical and Systems Engineering, Washington University in St. Louis, St. Louis
MO, 63130 USA e-mail: vignesh.narayanan@wustl.edu, wei.zhang@wustl.edu, jsli@wustl.edu.}
}

\maketitle
\begin{abstract}
	Controlling a large population, in the limit, a continuum, of structurally identical dynamical systems with parametric variations is a pervasive task in diverse applications in science and engineering. However, the severely underactuated nature and the inability to avail comprehensive state feedback information of such ensemble systems raise significant challenges in analysis and design of  ensemble systems. In this paper, we propose a moment-based ensemble control framework, which incorporates and expands the method of moments in probability theory 
	to control theory. In particular, we establish an equivalence between ensemble systems and their moment systems in terms of control and their controllability properties by extending the Hausdorff moment problem from the perspectives of differential geometry and dynamical systems. The developments enable the design of moment-feedback control laws for closing the loop in ensemble systems using the aggregated type of measurements. The feasibility of this closed-loop control design procedure is validated both mathematically and numerically.


\end{abstract}
 \section{Introduction}\label{sec: Introduction}

Large populations of uncoupled or interconnected dynamical systems are pervasive in diverse scientific domains, such as quantum science and technology \cite{Dong10,Li_PNAS11}, power systems \cite{bomela_18_tcl}, neuroscience \cite{ching_13_control,li_13_control}, emergent behaviors \cite{smale_07_emergents}, and robotics \cite{becker_12_robotics}. These population systems, formally referred to as \emph{ensemble systems}, generally exhibit variations in the parameters characterizing the dynamics of individual dynamic units in the ensemble. Such variations arise either by nature (e.g., different weight and size of birds in a flock), manufacturing (e.g., the variability in the fabrication resulting in different mass, friction coefficients, etc., of robots in a swarm), or design (e.g., the application of gradient fields resulting in Larmor frequency dispersion in magnetic resonance imaging). Owing to their prevalence in diverse emerging applications and rich mathematical structures, ensemble control problems have attracted significant attention and formed a new paradigm in systems and control over past years. Theoretically, the notions of ensemble controllability and observability have been introduced and extensively investigated, especially for linear \cite{li_10_ltv,brockett_2000_stochastic,Helmke14,Li_TAC16,Dirr18,wei_18_controllability, zeng_16_moment,zeng_15_ensemble,zeng_16_sampled,Li_SICON19}, bilinear \cite{li_09_bloch,Beauchard10,chen_19_controllability,chen_20_ensemble,chen_19_structure,Li_TAC19,Li_SICON20_arXiv}, and some classes of nonlinear systems \cite{chen_19_structure}. Computationally, various numerical algorithms have been proposed for synthesizing robust and optimal control signals to steer ensemble systems between desired states \cite{Li_PNAS11,Li_ACC12_SVD,Gong16,Li_SICON17,Li_Automatica18,wei_15_uniform,zlotnik_16_phase,zlotnik_12_neuron,kuritz_18_ensemble,zeng_18_computation,chen_18_OT}.

The fundamental challenges associated with ensemble control problems lie in the underactuated nature and the lack of comprehensive state feedback information of the entire ensemble. These bottlenecks have naturally pushed the research in ensemble control theory towards the direction of pursuing open-loop and sparsely distributed control scenarios. 
However, in many cutting-edge applications involving the control of ensemble systems, such as neuronal networks, robot swarms, spin ensembles, and cellular oscillators, aggregated type of spatially sparse measurements, such as coarse or fragmented images and partial snapshots, can be obtained \cite{hong_19_novel,couceiro_12_low,quantum_measurement_prl_19,kuritz_18_ensemble}. The availability of such measurements then opens up the possibility for utilizing population-level feedback to close the loop in ensemble control systems. 
In this paper, we adopt 
{the idea of statistical moments in probability theory to 
develop a moment-based ensemble control framework.} 

Approaches based on the use of ``moments'' have been introduced in systems theory, especially widely adopted in the context of stochastic control \cite{brockett_76_parametrically,yang_15_shapingcrowds}. On the other hand, several efforts have been made towards developing dynamic models and analyzing collective behaviors of deterministic populations using statistical moments \cite{zhang_18_modellingCollectBehavrMoment, singh_10_approximateChemical}. For example, an original work on controlling a homogeneous population of dynamical agents using their population density was proposed in \cite{brockett_12_notes}, and was later extended to control linear ensemble systems using their mean and variance through output feedback \cite{dirr_16_controlling} and, further, to analyze ensemble controllability and observability for time-invariant linear ensemble systems \cite{zeng_16_momentdyn, zeng_16_moment}. 

In contrast to the existing literature, the moment-based ensemble control framework proposed in this paper focuses on the analysis and manipulation of ensemble systems through the systems describing the dynamics of their moments induced by aggregated measurements. In particular, we categorize such measurements into two types - labeled and unlabeled, commonly found in practice, and associate ensemble systems with two different notions of moments, i.e., ensemble- and output-moments, respectively. By extending the classical moment problem from a geometric aspect, we establish a dynamic connection between an ensemble and its respective moment system, so that controlling an ensemble can be achieved by controlling its moment system. In addition, we show that both the ensemble system and the corresponding moment system share the same controllability property. Such equivalences enable and facilitate the design of closed-loop moment-feedback control laws by using aggregated measurements.

The paper is organized as follows. In the next section, we provide the background of ensemble control theory, and then introduce the notion of aggregated measurements to motivate the moment-based framework through the lens of probabilistic interpretations of deterministic ensemble systems. In Section \ref{sec:moment-space}, we review and extend the classical Hausdorff moment problem from a differential geometric viewpoint, which lays the foundation of our development. In Sections \ref{sec:Method_canonical} and \ref{sec:Method_general moment}, we develop the moment-based ensemble control approaches for the labeled and the unlabeled cases, respectively. 
In Section \ref{sec:Examples}, the estabilished framework is utilized to design moment-feedback control laws for closing the loop in ensemble systems, and the feasibility of this control design method is validated by numerical similations.

\section{Ensemble system and its probabilistic interpretation via aggregated measurements}\label{sec: Background}
\subsection{Ensemble systems}
An \emph{ensemble system} is a parameterized population of dynamical systems defined on a common manifold $M\subseteq\mathbb{R}^n$ of the form 
  \begin{equation}
  \label{eq:ensemble_dyn_generic}
 \frac{d}{dt}x(t,\beta) = F(x(t,\beta), \beta, u(t)), 
 \end{equation}
where $\beta$ is the parameter varying on $\Omega\subset\mathbb{R}^d$.
The parameter space $\Omega$ is generally assumed to be compact and it can be finite, countable or uncountable \cite{li_10_ltv,li_09_bloch}. Thus, the state-space of the ensemble system in \eqref{eq:ensemble_dyn_generic} is a space of $M$-valued functions defined on $\Omega$,  denoted by $\mathcal{F}(\Omega,M)$. The ensemble control problem deals with the design of a parameter-independent control input $u(t)\in\mathbb{R}^l$ which steers the ensemble system from an initial profile $x_0\in\mathcal{F}(\Omega,M)$ to a desired final profile $x_F\in\mathcal{F}(\Omega,M)$ in a finite time $T$. To understand ensemble control systems, we first introduce the notion of ensemble controllability, which describes the ability of an ensemble control law to manipulate the entire ensemble on its state-space in an approximating sense.

 \begin{definition}[Ensemble controllability]
 \label{def:ensmeble_controllability}
 The system in \eqref{eq:ensemble_dyn_generic} is said to be \emph{ensemble controllable} on $\mathcal{F}(\Omega,M)$, if for any $\varepsilon>0$ and starting with any initial profile $x_0\in\mathcal{F}(\Omega,M)$, there exists a control law $u(t)$ that steers the system into an $\varepsilon$-neighborhood of a desired target profile $x_F\in\mathcal{F}(\Omega,M)$ in a finite time $T>0$, i.e., $d(x(T,\cdot),x_F(\cdot))<\varepsilon$, where $d:\mathcal{F}(\Omega,M)\times\mathcal{F}(\Omega,M)\rightarrow\mathbb{R}$ is a metric on $\mathcal{F}(\Omega,M)$.
\end{definition}
Note that ensemble controllability is defined in the sense of approximation so that the topology, or equivalently the metric $d$, plays a crucial role in the study of ensemble systems. Also note that, in Definition \ref{def:ensmeble_controllability}, the final time $T$ may depend on the approximation error $\varepsilon$. 
To motivate the moment-based framework for tackling ensemble control problems, in the next section, we formally introduce the concepts of labeled and unlabeled aggregated measurements for ensmeble systems as well as their distinction.

\subsection{Aggregated measurements and labels of ensemble systems}\label{subsec:aggre_measurements}
In many emerging applications, though placing dedicated sensors to monitor the evolution of the states of individual systems in a population is not feasible, an aggregated, population-level measurement can be obtained (see Fig. \ref{fig:aggregated_neuron_spikes}). We refer to these types of measurements as \emph{aggregated measurements}, which are formally introduced in Definition \ref{def:aggregated_measurements}. For instance, in a typical neural recording \cite{hong_19_novel}, the spiking activity of a population of neurons can be recorded as a whole. In these recordings, the number of data-points acquired varies with each measurement instant, and in some cases, each data-point is anonymized, i.e., a data-point corresponding to a neural spike, and the neuron (system) in the population which generated it, may not be identified together. In this context, we refer to the parameter $\b$ in \eqref{eq:ensemble_dyn_generic} as a \emph{label} for the ensemble system. We then categorize an aggregated measurement as \emph{unlabeled} when the data-points in the measurement acquired cannot be associated with the specific system in the ensemble that generated it. Alternatively, when it is possible to label the data \cite{couceiro_12_low,kuritz_18_ensemble} and associate the data with the system that generated it, we refer to the measurement as \emph{labeled} measurement data. 
\begin{definition}[Aggregated measurements]
\label{def:aggregated_measurements}
Given an ensemble system as in \eqref{eq:ensemble_dyn_generic}, a \emph{labeled aggregated measurement} of the system at time $t$ is
a set of 2-tuples composed of the system labels and observations given by
\begin{align}\label{eq:aggregated_measurement_labeled}
Y_l(t):= \{(\b, h(\beta,x(t,\beta)))\mid \beta\in \Omega_t\},
\end{align}
and an \emph{unlabeled aggregated measurement} of the system at time $t$ is a set of its observations given by 
\begin{align}\label{eq:aggregated_measurement_unlabeled}
Y_u(t):= \{h({\beta},x(t,\beta))\mid \beta\in \Omega_t\},
\end{align}
where $\Omega_t \subseteq \Omega$, 
and ${h:\Omega\times M} \to \mathbb{R}^r$ is an observer or output function of the system.
\end{definition}
 \begin{figure}[ht]
     \centering
     \includegraphics[width=0.9\linewidth,keepaspectratio]{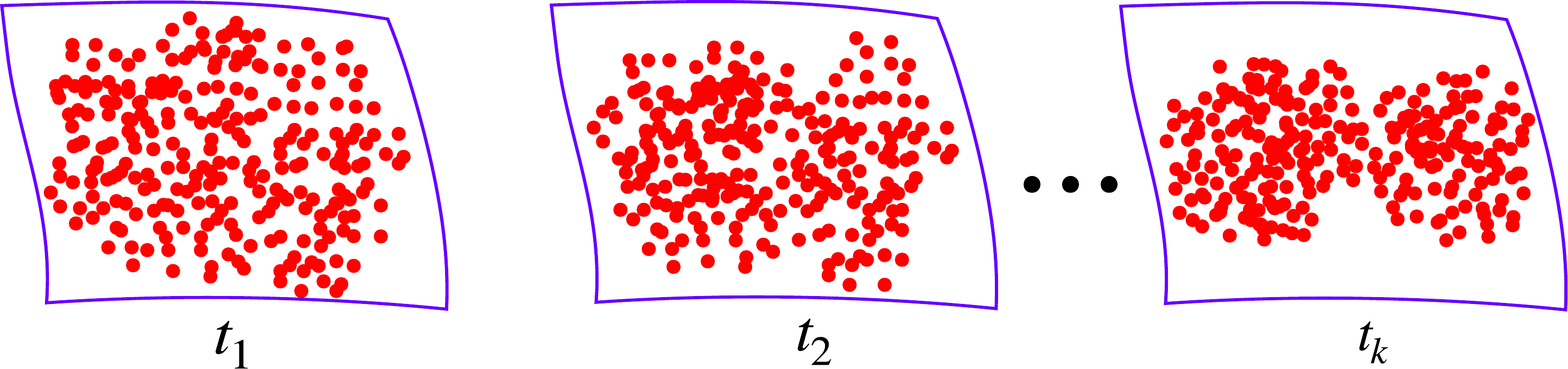}
     \caption{\footnotesize \noindent Illustration of aggregated measurements. The snapshots of the observed variables at different time instants may includes varying number of observations (represented using the time-dependent set $\Omega_t$) and these measurements, typically, cannot be associated with any particular system in the ensemble.}
     \label{fig:aggregated_neuron_spikes}
 \end{figure}
The availability of aggregated measurements in diverse application domains, while providentially paving the way to help close the feedback control loop for ensemble systems, introduces some challenges. In particular, the set $\Omega_t$ is typically a proper, finite subset of $\Omega$, and varying with time. As a result, the ensemble control problem cannot be approached in a traditional manner through classical state-feedback. 
In this context, effectively utilizing the aggregated measurements to close the feedback loop in ensemble systems necessitates a systematic framework, which has to be robust to the challenges introduced by the labeled and unlabeled aggregated measurements. To mitigate these challenges and establish such a rigorous framework, in this paper, we focus on instituting a connection between the ensemble system and the concept of statistical moments.

\subsection{Probabilistic interpretations of deterministic ensemble systems}
The method of moments is commonly used in the theory of probability and statistics for the purpose of characterizing probability distributions of random variables by sequences of numbers, called the moments of the random variables. Essentially, by regarding random variables as measurable functions on spaces with probability measures, this pertains to employing techniques developed for infinite sequences and series to analyze functions and measures.
Since ensemble systems are defined on function spaces, the method of moments can be integrated into ensemble control problems, which also gives rise to probabilistic interpretations of ensemble systems. 

Specifically, for an ensemble system indexed by the label $\beta$ as in \eqref{eq:ensemble_dyn_generic}, we can treat $\beta$ as a random variable taking values on the parameter space $\O$. In this way, the state variable $x(t,\beta)$ can interpreted in two different ways: deterministically, we consider $x(t,\beta)$ as the probability density function (with respect to the Lebesgue measure) of $\beta$ so that the ensemble system characterizes the dynamics of the probability distribution of $\beta$; probabilistically, we can view $x(t,\beta)$ as a random variable with $\beta$ denoting an outcome in the sample space $\Omega$, in which case, 
the solution of each individual system in this ensemble is a sample path of the stochastic process $\{x(t,\beta):t\geq0\}$. These two different viewpoints allow for formulating the ensemble control problem for the labeled and unlabeled cases, respectively, and also lead to two different notions of moments for ensemble systems, the ensemble and output moments, with respect to the system label $\beta$ and the state $x(t,\cdot)$, respectively. Specifically, the ensemble moment characterizes the ensemble system on a microscopic level, reflecting the dynamics of each individual system. On the contrary, the output moment describes the ensemble on a macroscopic level, concerning only with the overall distribution of the ensemble system on the state-space of functions defined on $\Omega$.

The foundation of the  moment method for ensemble systems is laid by the Hausdorff moment problem, since we consider $\O$ to be a compact space. In the next section, we conduct a brief survey of the Hausdorff moment problem, and extend some of the existing results for this problem to bridge the gap between the  moment and the ensemble control problems. In particular, we shed some light on the geometric characterization of the space of moment sequences to establish a link between the state-space of the ensemble system and the space of ensemble moments.

\section{Geometry of Moment Spaces}\label{sec:moment-space}
The intimacy between differential geometry and control theory dates back to the early 1970s, when techniques of differential calculus on manifolds were introduced to investigate fundamental properties of control systems, especially nonlinear systems \cite{brockett_14_early}. Since the state-spaces of ensemble systems are generally infinite-dimensional spaces, it is inevitable to involve Banach manifold techniques in the study of such systems. In order to initiate such an analysis, in this section, we introduce the Hausdorff moment problem from a geometric perspective, and in particular, to adopt its results for the ensemble systems, we extend them to the Banach space of $L^p$-functions, which lays the foundation for the moment-based ensemble control framework. 

\subsection{Hausdorff moment problem and its extension to the $L^p$-space}
\label{sec:Hausdorff_moment_problem}
The moment problem was first systematically formulated by Steiljits as the problem of finding a probability distribution on the interval $[0,\infty)$, given the sequence of moments of the distribution, and it dates back to as early as 1873 \cite{akhiezer_65_classical}. This problem has attracted considerable attention and subsequently, has been thoroughly studied under different settings yielding several results including that of Hausdorff, and others. 

The Hausdorff moment problem, named after the German mathematician Felix Hausdorff, concerns with the existence and uniqueness of a signed Borel measure $\mu$ on a compact interval $[a,b]$, so that 
$$m_k=\int_a^b\beta^kd\mu(\beta).$$
holds for all $k\in\mathbb{N}$, where $\mathbb{N}$ denotes the set of natural numbers. In this case, the sequence $m=(m_0,m_1,\dots)$ is called the \emph{moments} of the \emph{representing measure} $\mu$. 

Note that it suffices to study the moment problem for the unit interval $[0,1]$, since the linear transformation $\psi(\b)=a+(b-a)\beta$ maps $[0,1]$ bijectively onto $[a,b]$ such that the measure $\mu$ on $[a,b]$ is uniquely determined by the measure $\mu'$ on $[0,1]$ as $\mu(B)=\mu'(\psi^{-1}(B))$ for any Borel subset of $[a,b]$, where $\mu$, denoted by $\mu=\psi_\#\mu'$, is called the \emph{pushforward measure} of $\mu'$ by $\psi$. As a result, it can be shown that $\int_a^bhd\mu=\int_0^1h\circ\psi d\mu'$ holds for any Borel measurable function $h:[a,b]\rightarrow\mathbb{R}$ \cite{Bogachev07}. Applying this to the function $h(\b)=\b^k$ yields a relation between the moment sequences $m'=(m_0',m_1',\dots)$ of $\mu'$ and  $m=(m_0,m_1,\dots)$ of $\mu$ as
\begin{align*}
m_k&=\int_a^b\b^kd\mu(\b)=\int_0^1[a+(b-a)\b]^kd\mu'(\b)\\
&=\sum_{i=0}^k{k\choose i}a^{k-i}(b-a)^{i}m'_i,\quad k\in\mathbb{N}.
\end{align*}

To solve the moment problem, Hausdorff, in his work, introduced a sequence of difference operators $\Delta^0$, $\Delta^1$, $\dots$ for the moment sequence $m=(m_0,m_1,\dots)'$ iteratively as
$\Delta^0m_k=m_k$ and  $\Delta^nm_k=\Delta^{n-1}m_k-\Delta^{n-1}m_{k+1}$.
The explicit expression of $\Delta^n$ can then be revealed by induction on $n$, which follows
\begin{align*}
\Delta^n m_k=\sum_{i=0}^n{n \choose i}(-1)^im_{k+i},\quad k,n\in\mathbb{N}.
\end{align*}
The solution of the Hausdorff moment problem can then be represented in terms of these difference operators as follows: \emph{Given a sequence of real numbers $m=(m_0,m_1,\dots)$, there exists a signed Borel measure $\mu$ such that $m_k=\int_0^1\beta^kd\mu(\b)$ for all $m\in\mathbb{N}$ if and only if $\sum_{k=0}^n\big|{n \choose k}\Delta^{n-k}m_k\big|\leq C$ for all $n\in\mathbb{N}$ and some constant $C$ independent of $n$}, which was proposed and proved by Hausdorff in 1932 \cite{hausdorff_23_momentprobleme}.

Following Hausdorff's work, Hildebrandt and Schoenberg extended the moment problem to the multi-dimensional variable case in 1933 \cite{Hildebrandt33}. Specifically, if $\beta=(\beta_1,\dots,\beta_d)$ takes values in a compact subset $\Omega\subset\mathbb{R}^d$, then the moment sequence is parameterized by a multi-index $k=(k_1,\dots,k_d)\in\mathbb{N}^d$. In particular, we define the notation $\b^k$ to be the monomial $\b^k=\b_1^{k_1}\cdots\b_d^{k_d}$ in $(\b_1,\dots,\b_d)$ of degree $|k|=k_1+\cdots+k_d$, and correspondingly, the moment sequence $m_k=\int_\Omega\beta^kd\mu(\b)$ for a Borel measure on $\Omega$ is a multi-sequence. To illustrate the main idea, we consider the case $d=2$ with $\Omega=[0,1]\times[0,1]=[0,1]^2$ the unit square, 
and extend the definition of the difference operators to double sequences as $\Delta^nm_k=\Delta_1^{n_1}\Delta_2^{n_2}m_{k_1k_2}$, where $\Delta_i^{n_i}$ is the one-dimensional difference operator applied to the index $k_i$ for each $i=1,2$. 

\begin{lem}
\label{lem:Hausdorff_2d}
Given a double sequence of real numbers $m_k$, $k=(k_1,k_2)\in\mathbb{N}^2$, there exists a signed Borel measure $\mu$ on $\O=[0,1]^2$ such that $m_k=\int_\O\beta^kd\mu(\b)$
for all $k\in\mathbb{N}^2$ if and only if
\begin{align}
\label{eq:Hausdorff_2d}
\sum_{k_1,k_2=0}^{n}\Big|{n \choose k_1}{n \choose k_2}\Delta_1^{n-k_1}\Delta_2^{n-k_2}m_{k_1k_2}\Big|\leq C
\end{align}
for all $n\in\mathbb{N}$ and some constant $C$ independent of $n$. 
\end{lem}
\begin{proof}
See \cite{Hildebrandt33}.
\end{proof}

In the context of ensemble control problems considered in this work, we are particularly interested in ensemble systems defined on the Banach space of $L^p$-functions  $L^p(\O,\mathbb{R}^n)=\{\varphi:\O\rightarrow\mathbb{R}^n\mid\|\varphi\|_{p}<\infty\}$, where $\|\varphi\|_p=\big(\int_\O |\varphi|^pd\lambda\big)^{1/p}$ for $1\leq p<\infty$ and $\|\varphi\|_{\infty}=\inf\big\{a>0\mid\lambda(\{\beta\in\Omega\mid|\varphi(\b)|>a\})=0\big\}$
is the $L^p$-norm of $\varphi$, $|\cdot|$ is the Euclidean norm on $\mathbb{R}^n$, and $\lambda$ is the Lebesgue measure on $\Omega$. For the case $n=1$, it is a customary to denote the space of real-valued $L^p$-functions defined on $\O$ by $L^p(\O)$. In the following, to tailor the Hausdorff moment problem to this $L^p$-space setting, we would like to study moment sequences of Borel measures $\mu$ that are absolutely continuous with respect to the Lebesgue measure $\lambda$, and hence $d\mu=\varphi d\lambda$ holds for some $\varphi\in L^p(\O,\mathbb{R}^n)$ \cite{folland_13_real}. In the following, we extend the results in Lemma \ref{lem:Hausdorff_2d} to derive the necessary and sufficient condition for the existence and uniqueness of the real-valued $L^2$-solution of the Hausdorff moment problem.

\begin{prop}
\label{prop:Hausdorff_2d} A double sequence of real numbers $m_k$ indexed by $k=(k_1,k_2)\in\mathbb{N}^2$ is the moment sequence of a function $\varphi\in L^2(\O)$, i.e., 
\begin{align}\label{eq:moment_canonical_Prop_1}
m_k=\int_\O\beta^k\varphi(\b)d\b,    
\end{align}
if and only if 
\begin{align}
\label{eq:Hausdorff_2d}
(n+1)\sum_{k=0}^{n}\Big[{n \choose k}\Delta^{n-k}m_k\Big]^2\leq C
\end{align}
for all $n=(n_1,n_2)\in\mathbb{N}^2$ and some $C$ independent of $n$, where $(n+1)=(n_1+1)(n_2+1)$ and ${n \choose k}={n_1 \choose k_1}{n_2 \choose k_2}$. Moreover, the function $\varphi\in L^2(\O)$ representing the moment sequence $m_k$ is unique.
\end{prop}
\begin{proof}
See Appendix \ref{appd:Hausdorff_2d}.
\end{proof}

The results presented in Lemma \ref{lem:Hausdorff_2d} and Proposition \ref{prop:Hausdorff_2d} can be directly extended to the $d$-dimensional case with $d>2$ by induction. In the next section, these results will be utilized to study the space of moments associated with $L^p$-functions, and in particular, to derive some important properties of this space that are relevant to ensemble control problems. 
\subsection{Geometric properties of the moment space}\label{sec:geometry_moment}
Let $\mathcal{M}_p$ denote the space of moment sequences of real-valued $L^p$-functions defined on the $d$-dimensional cube $\Omega=[0,1]^d$. 
Then, Proposition \ref{prop:Hausdorff_2d} gives rise to an explicit description of the space $\mathcal{M}_2$ and also reveals an one-to-one correspondence between elements in $L^2(\O)$ and $\mathcal{M}_2$. Note that because $\Omega$ is compact, it has finite Lebesgue measure, which implies $L^p\subset L^2$ for all $2<p\leq\infty$ \cite{folland_13_real}. As a result, for each $2<p\leq\infty$, the space $\M_p$ can be obtained by restricting $\M_2$ to the subspace consisting of moment sequences for $L^p$-functions defined on $\O$, and hence, the one-to-one correspondence is also inherited by elements between $L^p(\O)$ and $\M_p$. 

\begin{lem}
\label{lem:L}
The space $\mathcal{M}_p$ consisting of moment sequences of functions in $L^p(\O)$ is a vector space over $\mathbb{R}$ for all $2\leq p\leq\infty$. Moreover, the map $\mathcal{L}:L^p(\O)\rightarrow\M_p$ assigning each function its moment sequence is a vector space isomorphism. 
\end{lem}
\begin{proof}
To show that $\M_p$ is a vector space, we pick  moment sequences $m^1,m^2\in\mathcal{M}_p$ representing $\varphi_1,\varphi_2\in L^p(\O)$, respectively, then $m^i_k=\int_\O\b^k\varphi_i(\b)d\b$ holds for all  $i=1,2$ and $k\in\mathbb{N}^d$, where $m^i_k$ denotes the $k^{\rm th}$ moment of $\varphi_i$, or equivalently, the $k^{\rm th}$ component of $m^i$. The linearity of integrals implies $c_1m^1_k+c_2m^2_k=\int_\O\b^k(c_1\varphi_1(\b)+c_2\varphi_2(\b))d\b$ for all $k\in\mathbb{N}^d$ and $c_1,c_2\in\mathbb{R}$, i.e., $c_1m^1+c_2m^2$ is the moment sequence of $c_1\varphi_1+c_2\varphi_2\in L^p(\O)$. This concludes that $\M_p$ is a vector space over $\mathbb{R}$ and $\L$ is a homomorphism. The one-to-one correspondence between elements in $L^p(\O)$ and $\M_p$ revealed by the Hausdorff moment problem then 
shows that $\L$ is bijective, and hence, a vector space isomorphism.
\end{proof}



To investigate the geometric properties of $\M_p$ that facilitate the analysis of control systems evolving on it, it is necessary to associate $\M_p$ with a topology. In particular, because $\L:L^p(\O)\rightarrow\M_p$ is surjective, the quotient topology generated by $\L$, i.e., the strongest topology on $\M_p$ such that $\L$ is continuous, is a good candidate for a topology on $\M_p$, which enables the study of geometric and topological properties of $\M_p$ through those of $L^p(\O)$. 


\begin{thm}
\label{thm:L_isometry}
Endowed with the quotient topology generated by $\L:L^p(\O)\rightarrow\M_p$, $\M_p$ is a Banach space, and $\L$ is an isometric isomorphism. Moreover, with the smooth structures on $L^p(\O)$ and $\M_p$ induced by themselves as Banach manifolds, $\L$ is a diffeomorphism.
\end{thm}
\begin{proof}
Since $\L$ is bijective and $L^p(\O)$ is a Banach space, the quotient topology on $\M_p$ generated by $\L$ is completely metrizable with a norm given by $\|m\|_p=\|\L^{-1}m\|_p$ for any $m\in\M_p$, where we also use the notation $\|\cdot\|_p$ to denote the norm on $\M_p$ inherited from $L^p(\O)$. As a result, $\M_p$ becomes a Banach space isometrically isomorphic to $L^p(\O)$.

Geometrically, regarding $L^p(\O)$ and $\M_p$ as Banach manifolds modeled on themselves, the linearity and continuity of $\L$ further imply the smoothness of $\L$, and hence $\L$ is a diffeomorphism. 
\end{proof}

The diffeomorphic property of $\L$ shown in Theorem \ref{thm:L_isometry} implies that its pushforward (differential) 
$$\L_*:TL^p(\O)\rightarrow T\M_p$$
is well-defined and is a diffeomorphism as well, where $TL^p(\O)$ and $T\M_p$ denote the tangent bundles of $L^p(\O)$ and $\M_p$, respectively. Algebraically, the linearity of $\L$ gives rise to an explicit expression of $\L_*=d\L$ as follows. Given any $\varphi\in L^p(\O)$, we identify the tangent space $T_\varphi L^p(\O)$ of $L^p(\O)$ at the point $\varphi\in L^p(\O)$ with $L^p(\O)$ so that every tangent vector $\psi\in T_\varphi L^p(\O)$ can also be considered as a function in $L^p(\O)$. Then, the differential $d\L_\varphi$ of the map $\mathcal{L}$ at $\varphi \in L^p(\O)$ satisfies $d\L_\varphi\psi=\L\psi$, i.e., $d\L_\varphi$ also maps every $L^p$-function to its moment sequence 
that is similarly identified with a tangent vector in $T_{\L\varphi}\M_p$.

We have initiated a detailed investigation into the Hausdorff moment problem for real-valued $L^p$-functions with $p\geq2$ from the geometric viewpoint. To deal with $n$-dimensional ensemble systems by using the moment method, it is inevitable to extend our discussion to functions in $L^p(\O,\mathbb{R}^n)$, which is a direct consequence of Theorem \ref{thm:L_isometry} as shown in the following corollary.

\begin{cor}\label{cor: moment-space and L-map}
The spaces $L^p(\O,\mathbb{R}^n)$ and $\M^n_p$ are isometrically isomorphic as vector spaces and diffeomorphic as Banach manifolds, where $\M_p^n$ denotes the product of $n$ copies of $\M_p$.
\end{cor} 
\begin{proof}
The space $L^p(\O,\mathbb{R}^n)$ can be naturally identified with $L^p(\O)\times\cdots\times L^p(\O)$, the product of $n$ copies of $L^p(\O)$, so that every element $\varphi\in L^p(\O,\mathbb{R}^n)$ admits a representation as an $n$-tuple $\varphi=(\varphi_1,\dots,\varphi_n)$ with $\varphi_i\in L^p(\O)$ for each $i$. By Theorem \ref{thm:L_isometry}, because each component $L^p(\O)$ in $L^p(\O,\mathbb{R}^n)$ is isometrically isomorphic and diffeomorphic to $\M_p$, the same properties also hold between  $L^p(\O,\mathbb{R}^n)$ and  $\M_p^n$.
\end{proof}

The proof of Corollary \ref{cor: moment-space and L-map} also gives rise to the construction of the diffeomorphism between $L^p(\O,\mathbb{R}^n)$ and $\M_p^n$. Specifically, for any $\varphi=(\varphi_1,\dots,\varphi_n)\in L^p(\O,\mathbb{R}^n)$, let $\L_i:L^p(\O)\rightarrow\M_p$ denote the map assigning the $i^{\rm th}$-component $\varphi_i$ of $\varphi$ to its moment sequence. Then, the desired diffeomorphism, also denoted by $\L:L^p(\O,\mathbb{R}^n)\rightarrow\M_p^n$, is given by $(\varphi_1,\dots,\varphi_n)\mapsto(\L_1\varphi_1,\dots,\L_n\varphi_n)$.

Having established the properties of the moment space $\M_p^n$ in this section, we are now equipped to formally build the connection between the ensemble control and moment problems, which is the object of the following sections. In particular, in the next section, we focus on the ensemble control problem that arises in the context of labeled aggregated measurements.

\section{Ensemble control with labeled Aggregated measurement}\label{sec:Method_canonical}
{The establishment of the relationship between the spaces of moment sequences and $L^p$-functions in the preceding section provides the necessary tool for studying ensemble systems through the dynamics of the moment sequences representing the states of the ensemble systems.
To start with, we focus on the case of ensemble systems and the ensemble moments that are relevant in the context of labeled aggregated measurements. In particular, we rigorously introduce the notion of ensemble moments, and then derive the systems governing the dynamics of the ensemble moments. Utilizing the moment problem from a dynamical systems perspective, we formalize the connection between ensemble control systems and the associated moment control systems.}

\subsection{Ensemble moments}
To develop moment-based methods for tackling ensemble control tasks such as ensemble controllability analysis and ensemble control law design, we begin by defining the notion of ensemble moments as follows.
\begin{definition}[Ensemble moments]
\label{defn:moments}
Consider the ensemble system in \eqref{eq:ensemble_dyn_generic} defined on $\mathcal{F}(\Omega,M)$. Let $h:\Omega\times M\rightarrow\mathbb{R}^r$ be an output function of the system, then the $k^{\rm th}$ ensemble moment of the system is defined by
$$m_k(t)=\int_{\Omega}h(\beta,x(t,\beta))^{k}d\mu_t(\beta),$$
where $k=(k_1,\dots,k_r)\in\mathbb{N}^r$ is a multi-index, $h^{k}$ denotes the monomial $h_1^{k_1}\cdots h_r^{k_r}$ of degree $|k|=k_1+\cdots+k_r$ in $h=(h_1,\dots,h_r)$, and $\mu_t$ is a time-dependent \emph{signed} Borel measure on $\Omega$.
\end{definition}

In the case that $h$ is a smooth output function and $\Omega$ is compact with $\mu_t$ a finite measure for all $t$, $m_k(t)<\infty$ holds for all $t\in\mathbb{R}$ and $k\in\mathbb{N}^d$, so that ensemble moments are always  well-defined. In this case, it is possible to study ensemble systems by analyzing the dynamics of their ensemble moments. In particular, we want to understand whether ensemble control problems can be solved by manipulating the systems governing the dynamics of the ensemble moments. To this end, we will expand the Hausdorff moment problem introduced in the previous section to a dynamic setting to study the connection between ensemble systems and their moment systems.

\subsection{Ensemble system and dynamic-Hausdorff moment problem}
To illuminate the idea of the moment-based ensemble control framework, we will restrict our focus to control-affine ensemble systems evolving on the Banach space $L^p(\Omega,\mathbb{R}^n)$ for some $p\geq2$
of the form
  \begin{align}
  \label{eq:ensemble_affine}
 \frac{d}{dt}x(t,\beta) = f(x(t,\beta),\beta) + \sum_{i=1}^l u_i(t)g_i(x(t,\beta),\beta),
\end{align}
where $\beta$ denotes the system label taking values on the $d$-dimensional cube $\O=[0,1]^d$.
Correspondingly, the drift and control vector fields $f(x(t,\b),\b)$ and $g_i(x(t,\b),\b)$, regarded as functions in $\b$, also lie in $L^p(\O,\mathbb{R}^n)$. Moreover, to guarantee the existence and uniqueness of the solution of the system, we also require $f(x(t,\b),\b)$ and $g(x(t,\b),\b)$ to be continuously differentiable in $x(t,\b)$. The $L^p$-space setting then enables the definition of $L^p$-ensemble controllability for the system in \eqref{eq:ensemble_affine}, that is, the metric $d$ in Definition \ref{def:ensmeble_controllability} is taken to be the $L^p$-distance: $\|\varphi-\psi\|_p$ for any $\varphi,\psi\in L^p(\O,\mathbb{R}^n)$. 
Further, in this section, {we can consider $h(\beta,x(t,\beta))=\beta$ and time-dependent measure $d\mu_t(\beta)=x(t,\beta)d\beta$}. As a result, the $k^{\rm th}$ ensemble moment is given by
\begin{align}
\label{eq:classical_moments}
m_k(t)=\int_\Omega\beta^{k}x(t,\beta)d\beta, \quad k\in\mathbb{N}^d,
\end{align}
or equivalently from Corollary \ref{cor: moment-space and L-map}, $m(t)=\L x(t,\beta)$, where $m(t)\in\M_p^n$ denotes the moment sequence of $x(t,\cdot)$. 

Before formally introducing the moment system associated with a general control-affine ensemble system as in  \eqref{eq:ensemble_affine}, in the following, we first use some examples, including the bilinear Bloch ensemble and a nonlinear ensemble, to illustrate the computation of their moment systems.

\begin{ex}\label{ex:Bloch}
Consider an ensemble of Bloch systems
\begin{align}\label{eq:bloch}
\frac{d}{dt} M(t,\epsilon)=\epsilon[u(t)\O_y+v(t)\O_x]M(t,\epsilon)
\end{align}
describing the dynamics of a sample of spin-$\frac{1}{2}$ nuclei immersed in a static magnetic field, where the state $M(t,\epsilon)$ denotes the bulk magnetization, 
\begin{align*}
	&\Omega_x =\left[\begin{array}{ccc} 0 & 0 & 0 \\ 0 & 0 & -1 \\ 0 & 1 & 0 \end{array}\right] \quad\text{and}\quad \Omega_y=\left[\begin{array}{ccc} 0 & 0 & 1 \\ 0 & 0 & 0 \\ -1 & 0 & 0 \end{array}\right]
\end{align*}
are the generators of rotations around $x$- and $y$-axes, respectively, $u(t)$ and $v(t)$ are the external radio frequency (rf) fields as the control inputs to the system, and $\epsilon\in[1-\delta,1+\delta]$ with $\delta<1$ denotes the inhomogeneity of the rf fields. The $k^{\rm th}$ ensemble moment of the system in \eqref{eq:bloch} is given by $m_k(t)=\int_{1-\delta}^{1+\delta}\epsilon^k M(t,\epsilon)d\epsilon$, and its dynamics is governed by the differential equation
\begin{align}
\frac{d}{dt}{m}_k(t) & = \frac{d}{dt}\int_{1-\delta}^{1+\delta}\epsilon^k M(t,\epsilon)d\epsilon=\int_{1-\delta}^{1+\delta}\epsilon^k \frac{d}{dt}M(t,\epsilon)d\epsilon\nonumber\\ &=\int_{1-\delta}^{1+\delta}\epsilon^{k+1}\left[u(t)\Omega_y+v(t)\Omega_x\right]M(t,\epsilon)d\epsilon  \nonumber \\
& = \left[u(t)\Omega_y+v(t)\Omega_x\right]m_{k+1}(t),
\label{eq:bloch_moment_dyn}
\end{align}
where the change of the differentiation and integration in the second equality follows from the dominant convergence theorem \cite{folland_13_real}. Moreover, it is interesting to note that the moment system in \eqref{eq:bloch_moment_dyn} of the Bloch ensemble coincides with a chain network of countably many bilinear systems. 

\end{ex}
\begin{ex}\label{ex:Nonlinear_ensemble} 
Consider a nonlinear ensemble system with the dynamics given by 
\begin{align}
\label{eq:nonlinear_ex}
    \frac{d}{dt}z(t,\beta) = f(z,\beta)+B(\beta)u(t),
\end{align}
where $z(t,\cdot)=\left[\begin{array}{c} x(t,\cdot) \\ y(t,\cdot) \end{array}\right]$ is the state, $\b \in [0,1]$ is the system label, $f=\beta \left[\begin{array}{c} y \\ -y-\sin{x}\end{array}\right]$ and $B= \left[\begin{array}{c} 0 \\ \b\end{array}\right]$ are the drift and control vector fields, respectively. Similar to the previous example, the system governing the dynamics of the ensemble moments can be derived as 
\begin{align*}
    &\frac{d}{dt}m_k(t)=\int_0^1\b^k\frac{d}{dt}z(t,\b)d\b\\
    &=\int_0^1\b^{k+1} \left[\begin{array}{c} y(t,\b) \\ - y(t,\b)- \sin x(t,\b)+ u(t) \end{array}\right] d\b\\
    &=\int_0^1\b^{k+1} \left[\begin{array}{c} y(t,\b) \\ - y(t,\b)- \sin x(t,\b) \end{array}\right] d\b+\frac{u(t)}{k+1}\left[\begin{array}{c} 0 \\ 1\end{array}\right]
\end{align*}
In this case, the moment system cannot be expressed as a closed-form differential equation with the moment sequence being the state variable, but it is still in the control-affine form.
\end{ex}

Examples \ref{ex:Bloch} and \ref{ex:Nonlinear_ensemble} shed light on a general fact that for control-affine ensemble systems, their moment systems are also in the control-affine form no matter whether they yield closed-form expressions or not. This observation can be proved by using the geometric properties of moment sequences revealed in Section \ref{sec:geometry_moment}.

\begin{thm}\label{thm: affine to affine}
The dynamics of the ensemble moments of a control-affine system defined on $L^p(\O,\mathbb{R}^n)$ with $p\geq2$ as in \eqref{eq:ensemble_affine} is governed by a control-affine system on $\M_p^n$.
\end{thm}
\begin{proof}
The proof follows from the following calculations 
\begin{align}
\frac{d}{dt}m(t)&=\frac{d}{dt}\L x(t,\beta)=\L_*\Big(\frac{d}{dt}x(t,\beta)\Big)\nonumber\\
&=\L_*\Big(f(x(t,\beta),\beta)+\sum_{i=1}^mu_i(t)g_i(x(t,\beta),\beta)\Big)\nonumber\\
&=\L_*f(m(t))+\sum_{i=1}^mu_i(t)\L_*g_i(m(t)), 
\label{eq:moment_system}
\end{align}
where we exploit the chain rule and linearity of $\L_*$ in the second and last equalities, respectively.
\end{proof}

As discussed in Section \ref{sec:geometry_moment}, because $\L$ is a linear map, its differential $\L_*$ evaluated at every point in $L^p(\O,\mathbb{R}^n)$ can be identified with $\L$ itself. This property, together with the computation shown in the proof of Theorem \ref{thm: affine to affine}, further gives rise to an explicit expression of the moment system of the control-affine ensemble system in \eqref{eq:ensemble_affine}. Specifically, at any state $x(t,\cdot)\in L^p(\O,\mathbb{R}^n)$, the drift and control vector fields of the moment system in \eqref{eq:moment_system}, i.e., $\bar f:=\L_*f$ and $\bar g_i:=\L_*g_i$, are just the moment sequences of $f(x(t,\cdot),\cdot)$ and $g_i(x(t,\cdot),\cdot)$ as functions in $L^p(\O,\mathbb{R}^n)$, respectively. 



\begin{rem}
\label{rmk:same_control}
Due to the control-affine structure of the ensemble system in \eqref{eq:ensemble_affine}, the vector fields governing the system dynamics is linear in the control inputs $u_i$. Together with the linearity of $\L_*$ and parameter-independence of $u_i$, the control inputs $u_i$ remain the same in its moment system. In other words, the moment transformation for control-affine ensemble systems leaves the control inputs untouched, which opens up the possibility of tackling ensemble control tasks through controlling the corresponding moment systems.
\end{rem}

Motivated by the observation in Remark \ref{rmk:same_control} that control-affine ensemble systems and their moment systems share the same control inputs, in the next section, we focus on the study of control-related properties that are also preserved by the moment transformation, especially, the controllability. 
\subsection{Equivalence between ensemble and moment systems}

To discuss controllability, we first notice that the moment space $\M_p$ is an infinite-dimensional space so that moment systems defined on it are also infinite-dimensional systems. Therefore, controllability of moment systems should also be defined in the approximation sense as the definition of ensemble controllability (see Definition \ref{def:ensmeble_controllability}). Precisely, a system in the form of \eqref{eq:moment_system} on $\M_p$ is controllable if for any $\varepsilon>0$ and initial and desired final states $m_0$ and $m_F$, there exits a control law $u(t)$ that steers the system to an $\varepsilon$-neighborhood of $m_F$ in a finite time $T$, i.e., $\|m(T)-m_F\|_p<\varepsilon$. To distinguish it with ensemble controllability, we call it \emph{approximate controllability}. 

\begin{thm}
\label{thm:controllability}
The system in \eqref{eq:ensemble_affine} is $L^p$-ensemble controllable on $L^p(\O,\mathbb{R}^n)$ if and only if its moment system in \eqref{eq:moment_system} is approximately controllable on $\M_p^n$.
\end{thm}
\begin{proof}
The one-to-one correspondence between elements in the state-spaces $L^p(\O,\mathbb{R}^n)$ and $\M_p^n$ of the two systems revealed in Corollary \ref{cor: moment-space and L-map} implies that the ensemble system in \eqref{eq:ensemble_affine} is at the state $x(t,\cdot)\in L^p(\O,\mathbb{R}^n)$ if and only if its moment system in \eqref{eq:moment_system} is at $m(t)=\L x(t,\cdot)\in\M_p^n$. Then, due to the observation in Remark \ref{rmk:same_control} that these two systems share the same control inputs, it suffices to show that $\L$ maps any $\varepsilon$-ball in $L^p(\O,\mathbb{R}^n)$ to an $\varepsilon$-ball in $\M_p^n$ for every $\varepsilon>0$. However, this is just a direct consequence of the isometric property of $\L$ proved in Corollary \ref{cor: moment-space and L-map}. 
\end{proof}
\begin{rem}
As shown in Theorem \ref{thm:controllability}, the ensemble system in \eqref{eq:ensemble_affine} and its moment system in \eqref{eq:moment_system} share the same controllability property, although they may consist of different number of dynamic units. Especially, in the case that $\Omega$ is uncountable, the moment transformation yields a controllability-preserving model reduction of an ensemble 
system containing uncountably many 
individual systems to its moment system containing countably many components.
\end{rem}

The results presented in this section pertinent to the case when the measurements are labeled since the definition of ensemble moments in \eqref{eq:classical_moments} involves the system label $\beta$ explicitly. Though not uncommon in applications, extending the ideas presented in this section for the case with unlabeled measurements will be considered in the next section, which will encompass a lot more applications.
\section{Pattern Control Using Unlabeled Aggregated Measurements}\label{sec:Method_general moment}
Having established a strong relation in terms of the control-equivalence between the ensemble systems and their ensemble moment systems, we have devised a systematic framework to study the moment-based ensemble control problems with labeled aggregated measurements. In this section, we will extend the results developed in the previous section to 
the case that arises in the context of unlabeled aggregated measurements.
\subsection{Output moments and the moment system}
To relax the requirement of labels of measurement data in the moment-based ensemble control framework, it is necessary to use an output function $h$ depending on the system label $\beta$ only implicitly through the state variable $x(t,\beta)$, i.e., $h(\beta,x(t,\beta))=h(x(t,\beta))$. In addition, the definition of moments is then required to be independent of $\b$ as well. 

To this end, we define the notion of \emph{output moments} as 
\begin{align}\label{eq:moments_hk}
    \mathfrak{m}_k(t)=\int_{\mathbb{R}^n}h^k(x)d\nu_t(x)
\end{align}
for the ensemble system defined on  $L^p(\Omega,\mathbb{R}^n)$ in \eqref{eq:ensemble_affine}, where $h:\mathbb{R}^n\rightarrow\mathbb{R}^r$ is the output function of the system, $k=(k_1,\dots,k_r)\in\mathbb{N}^r$ is a multi-index, $\nu_t=(x_t)_\#\lambda$ is the pushforward of the Lebesgue measure $\lambda$ on $\O$ by the state $x_t(\cdot)=x(t,\cdot)$. As a result, the $k^{\rm th}$ output moment defined in \eqref{eq:moments_hk} satisfies
\begin{align}
\label{eq:pushforward_measure}
\int_{\mathbb{R}^n}h^kd\nu_t=\int_\Omega h^k\circ x_td\lambda.
\end{align}
Furthermore, if each individual system in the ensemble in \eqref{eq:ensemble_affine} is observable, it is possible to pick $h(x)=x$, the identity function, so that 
\begin{align}
\label{eq:output_moment}
\mathfrak{m}_k(t)=\int_{\mathbb{R}^n}x^kd\nu_t(x)={\int_\Omega x^k(t,\beta)d\beta.}
\end{align}
To guarantee that the output moments are well-defined, we restrict the state-space of the system to $L^{\infty}(\O,\mathbb{R}^n)$, in which case the pushforward measure $\nu_t=(x_t)_\#\lambda$ is supported in the range of $x_t$, more explicitly, inside the $n$-dimensional cube $[-\|x_t\|_\infty,\|x_t\|_\infty]^n$ centered at the origin with the side length $2\|x_t\|_\infty$.

\begin{rem}
\label{rmk:output_moment_probability}
Because the Lebesgue measure of $\O=[0,1]^d$ is 1, $(\Omega,\lambda)$ is a probability space. As a result, the state $x_t\in L^\infty(\Omega,\mathbb{R}^n)$ of the ensemble system can be interpreted as an $\mathbb{R}^n$-valued random variable on $(\Omega,\lambda)$. Moreover, the integral-preserving property of the pushforward operation guarantees that $\nu_t=(x_t)_\#\lambda$ is a probability measure on $\mathbb{R}^n$, which is the so-called 
(joint) probability distribution or law of $x_t$ in probability theory. In this case, the output moments $\mathfrak{m}_k(t)$ defined in \eqref{eq:output_moment} indeed coincide with the moments of the random variable $x_t$ in the sense of probability theory, and further, the solution of the ensemble system can be interpreted as a sample path of the stochastic process $\{x_t:t\geq0\}$.
\end{rem}

Recall that the output moments are defined using the output functions that are not explicitly related to $\b$. This ensures that the output moment sequence can be computed from unlabeled aggregated measurements. In principle, this loss of information 
about the system label will need to be taken into account while analyzing the output moments in the context of ensemble control, and this investigation is carried out in the next section.
\subsection{Output moment problem in the context of ensemble control}
Analogous to the classical Hausdorff moment problem, the output moment problem, which is considered here, focuses on the relationship between $L^\infty$-functions and their output moment sequences. In particular, we denote the space of output moment sequences of functions in $L^\infty(\Omega)$ by $\hat{\M}_\infty$, then output moment sequences of functions in $L^\infty(\Omega)$ are in $\hat{\M}_\infty^n$. 
However, different from ensemble moment sequences in $\M_\infty^n$, output moment sequences in $\hat\M_\infty^n$ do not correspond to functions in $L^\infty(\O,\mathbb{R}^n)$ in an one-to-one fashion. Specifically, different $L^\infty$-functions may have the same output moment sequence as shown in the following example.

\begin{ex}
\label{ex:output_moment}
Consider the indicator functions of $[0,1/2]$ and $[1/2,1]$ defined on $[0,1]$,  respectively, 
$$I_1(\beta)=\begin{cases}1,\ 0\leq\beta\leq\frac{1}{2}\\ 0,\ \frac{1}{2}<\beta\leq 1 \end{cases},\quad I_2(\beta)=\begin{cases}0,\ 0\leq\beta<\frac{1}{2}\\ 1,\ \frac{1}{2}\leq\beta\leq 1 \end{cases}.$$
The pushforwards $\nu_1=(I_1)_\#\lambda$ and $\nu_2=(I_2)_\#\lambda$ of the Lebesgue measure $\lambda$ on $[0,1]$ are measures supported on the discrete space $\{0,1\}$ satisfying $\nu_i(\{j\})=1/2$ for all $i=1,2$ and $j=0,1$, which implies $\nu_1=\nu_2$ and both of $I_1$ and $I_2$ follows the Bernoulli distribution. Consequently, $I_1$ and $I_2$ have the same output moment sequence as $\mathfrak{m}_0=1$ and $\mathfrak{m}_k=1/2$ for all $k=1,2,\dots$.
\end{ex}

In addition to {illustrating} the lack of one-to-one correspondence between elements in $\hat{\M}_\infty^n$ and $L^\infty(\O,\mathbb{R}^n)$, Example \ref{ex:output_moment} provides a hint regarding the output moment problem: Output moment sequences in $\hat\M_\infty^n$ are determined by measures pushforwarded by $L^\infty$-functions rather than $L^\infty$-functions themselves. 
Accordingly, we can consider a variation of the classical moment problem, namely, the output moment problem concerning with determining whether an output moment sequence uniquely represents a pushforward measure.
Notice that, for any function $\varphi\in L^\infty(\Omega,\mathbb{R}^n)$, its output moment sequence $\mathfrak{m}_k=\int_{\mathbb{R}^n}x^kd(\varphi_\#\lambda)$ is the moment sequence of the probability measure $\varphi_\#\lambda$, which is compactly supported as discussed above Remark \ref{rmk:output_moment_probability}, and hence, the Hausdorff moment problem is applicable to show the existence and uniqueness of the (probability) measure representing a given output moment sequence. Formally, let $$\mathcal{P}^n=\{\nu:\nu=\varphi_\#\lambda\text{ and }\varphi\in L^\infty(\O,\mathbb{R}^n)\}$$ denote the space of probability measures on $\mathbb{R}^n$ that are pushforwards of the Lebesgue measure $\lambda$ on $[0,1]^d$ by $L^\infty$-functions, then the output moment problem
gives an one-to-one correspondence between elements in $\mathcal{P}^n$ and $\hat\M_\infty^n$.



\begin{thm}
\label{thm:output_moment_problem}
The map $\hat{\mathcal{L}}:\mathcal{P}^n\rightarrow\hat\M_\infty^n$, assigning each probability measure in $\mathcal{P}^n$ its moment sequence in $\hat\M_\infty^n$, is bijective. 
\end{thm}
\begin{proof}
The surjectivity of $\hat{\mathcal{L}}$ directly follows from the fact that elements in $\mathcal{P}^n$ are pushforward measures by $L^\infty$-functions so that all orders of their moments are well-defined.

To prove the injectivity, it suffices to show that if $\varphi,\psi\in L^\infty(\Omega,\mathbb{R}^n)$ have the same output moment sequence $\mathfrak{m}\in\hat\M_\infty^n$, then they pushforward the Lebesgue measure $\lambda$ on $\Omega$ to the same measure on $\mathbb{R}^n$. Let $\varphi_i,\psi_i\in L^\infty(\Omega)$, and $k_i\in\mathbb{N}$ be the $i^{\rm th}$ components of $\varphi$, $\psi$, and $k$, respectively. Note that for the multi-index $k$ in the form of $k_i=2q$ for some $q\in\mathbb{N}$ and $k_j=0$ for $j\neq i$, we have {$\mathfrak{m}_k^{1/2q}=\|\varphi_i\|_{2q}=\|\psi_i\|_{2q}$.} By letting $q\rightarrow\infty$, we obtain $\|\varphi_i\|_\infty=\|\psi_i\|_\infty$ for each $i=1,\dots,n$, which implies $\|\varphi\|_\infty=\|\psi\|_\infty$. As a result, $\varphi_\#\lambda$ and $\psi_\#\lambda$ are supported on the same compact set. The application of the Hasudorff moment problem to Borel measures on this set then concludes $\varphi_\#\lambda=\psi_\#\lambda$. 
\end{proof}

To prepare the study of controllability of output moment systems, it is required to equip the space $\hat\M_\infty^n$ with a metric for the purpose of measuring the distance between two output moment sequences. To this end, notice that for any $\varphi\in L^\infty(\O,\mathbb{R}^n)$, its output moments satisfy
\begin{align*}
\mathfrak{m}_k=\int_\Omega \varphi^k(\beta)d\beta\leq\int_\Omega |\varphi^k(\beta)|d\beta\leq \|\varphi\|_\infty^{|k|},
\end{align*}
where $|k|=k_1+\cdots+k_n$ is the order of the multi-index $k=(k_1,\dots,k_n)\in\mathbb{N}^n$. This implies that $\mathfrak{m}_k^{1/|k|}\leq\|\varphi\|_\infty<\infty$ for all $k\in\mathbb{N}^n$, and hence the \emph{radical} of $\mathfrak{m}$, that is, the sequence ${\rm rad}(\mathfrak{m})$ with the $k^{\rm th}$ component $\mathfrak{m}_k^{1/|k|}$, belongs to  $\ell^\infty(\mathbb{N}^n)$, the Banach space of bounded $n$-sequences, with the $\ell^\infty$-norm $\|{\rm rad}(\mathfrak{m})\|_\infty\leq\|\varphi\|_\infty$. As a result, we can define the distance of two moment sequences $\mathfrak{m}$ and $\mathfrak{n}$ as the $\ell^\infty$-distance of their radicals, that is, $d(\mathfrak{m},\mathfrak{n})=\|{\rm rad}(\mathfrak{m})-{\rm rad}(\mathfrak{n})\|_\infty=\sup_{k\in\mathbb{N}}\big|\mathfrak{m}_k^{1/|k|}-\mathfrak{n}_k^{1/|k|}\big|$. Moreover, because of the one-to-one correspondence between probability measures in $\mathcal{P}^n$ and output moment sequences $\hat\M_\infty^n$ revealed in Theorem \ref{thm:output_moment_problem}, $\mathcal{P}^n$ also inherits the distance function $d$ from $\hat\M_\infty^n$ as $d(\mu,\nu)=d(\hat{\L}\mu,\hat{\L}\nu)$ for any $\mu,\nu\in\mathcal{P}^n$.

The detailed investigation into the output moment problem presented in this section provides the preparation that is needed for the study of output moment systems and their relationships with ensemble control systems, which is the main focus of the next section.

\subsection{Pattern formation for ensemble systems via output moments}
As discussed in the previous section, the output moment sequence of an ensemble system provides information about the overall distribution of the whole ensemble instead of the behavior of each individual system. This nature of output moments enables the utilization of them in the study of pattern formation problems for ensemble  systems. To systematically study this problem, we first introduce the notion of pattern controllability for ensemble systems.

\begin{definition}[Pattern controllability]
An ensemble system $\frac{d}{dt}x(t,\b)=F(x(t,\b),\b,u(t))$ defined on the function space $\mathcal{F}(\O,M)$, where the parameter $\beta$ takes values on the measurable space $(\O,\mu)$ with $\mu$ a Borel measure, is pattern controllable if for any Borel measure $\nu$ on $M$, initial state $x_0\in\mathcal{F}(\O,M)$, and $\varepsilon>0$, there exists a control law $u(t)$ that steers the system to a final state $x_T\in\mathcal{F}(\O,M)$ such that $d\big((x_T)_\#\mu,\nu\big)<\varepsilon$ 
in a finite time $T>0$. 
\end{definition}

Similar to ensemble controllability defined in Definition \ref{def:ensmeble_controllability}, pattern controllability is also defined in the sense of approximation, and in this case, the distribution of the ensemble represented in terms of the measure pushed forward by the state takes precedence over the profile of the ensemble system itself. Moreover, to be consistent with previous sections, we also primarily focus on control-affine ensemble systems defined on $L^{\infty}(\O,\mathbb{R}^n)$ of the form \eqref{eq:ensemble_affine}
with the parameter space $\Omega=[0,1]^d$ equipped with the Lebesgue measure $\lambda$. However, for such a system, due to the lack of one-to-one correspondence between elements in $L^\infty(\O,\mathbb{R}^n)$ and $\hat\M_\infty^n$, it is impossible to pushforward the vector fields on $L^\infty(\O,\mathbb{R}^n)$ governing the ensemble dynamics to some globally well-defined vector fields on $\hat\M_\infty^n$ governing the dynamics of its output moment sequence as presented in the proof of Theorem \ref{thm: affine to affine}. Fortunately, this output moment transformation for ensemble systems is still valid locally, which is summarized in the following result.

\begin{thm}
The dynamics of the output moment sequence of a control-affine ensemble system defined on $L^\infty(\O,\mathbb{R}^n)$ is a locally control-affine system on $\hat\M^n_\infty$.
\end{thm}
\begin{proof}
Let $\bar\L:L^\infty(\O,\mathbb{R}^n)\rightarrow\hat\M_\infty^n$ denote the map assigning each $L^\infty$-function its output moment sequence, then $\bar\L$ is surjective. Note that $L^\infty(\O,\mathbb{R}^n)$ and $\hat\M_\infty^n$ have the same dimension, then, together with the surjectivity of $\bar\L$, we know that $\bar\L$ must be a local diffeomorphism. Consequently, the same proof as Theorem \ref{thm: affine to affine} can be applied locally to show that the output moment system of a control-affine ensemble system is control-affine on any open subset of $\hat\M_\infty^n$ where $\bar\L$ is restricted to a diffeomorphism.
\end{proof}

As shown in Theorem \ref{thm:output_moment_problem}, from the perspective of probability theory and statistics, the output moment problem gives an one-to-one correspondence between probability distributions of $L^\infty$-random variables and their moment sequences. In the context of control and dynamical systems theory, this immediately leads to an equivalence between pattern controllability of ensemble systems and approximate controllability of their output moment systems. 

\begin{thm}
\label{thm:equivalence_output_moment}
An ensemble system defined on $L^\infty(\O,\mathbb{R}^n)$ is pattern controllable if and only if its output moment system is approximately controllable on $\hat\M_\infty^n$.
\end{thm}
\begin{proof}
The proof is the same as the proof of Theorem \ref{thm:controllability} by replacing $\L$ by $\hat\L$.
\end{proof}
Thus far, we have established the equivalence of controlling an ensemble system and its two moment systems in two different types of ensemble control tasks, respectively. 
Supported by this fundamental investigation, in the next section, we demonstrate the application of the proposed moment-based ensemble control framework to the design of ensemble control laws.
\section{Moment-Based Ensemble Control Design}\label{sec:Examples}
To incorporate the approximation idea of ensemble controllability into ensemble control law design tasks, we formulate ensemble control problems for systems with input-affine nonlinear dynamics as feedback stabilization problems for their associated moment systems. This further gives rise to a systematic feedback control design strategy for closing the loop of ensemble systems, which has been a missing element in ensemble control literature.

\subsection{Ensemble control through moment stabilization} \label{subsec:lyap}
Given a control-affine ensemble system of the form \eqref{eq:ensemble_affine} that is ensemble controllable on $L^p(\O,\mathbb{R}^n)$, let $m_F\in\mathcal{M}_p^n$ denote the moment sequence representing the desired target profile $x_F\in {L}^p(\Omega,\mathbb{R}^n)$, and $e(t) = m(t)- m_F$ be the error between the current and desired moment sequences. To steer the ensemble system arbitrarily close to the target state $x_F$ is equivalent to asymptotically stabilize the error $e(t)$ 
to $0$, and this stabilization problem can be accomplished by using the Lyapunov method. In particular, we may define a candidate Lyapunov function in the form of 
\begin{align*}
    V(t) = L(e(t)),
\end{align*}
which is differentiable with respect to $e(t)$ and satisfies $L(e(t))>0$ for $e(t)\ne 0$.
Taking the time-derivative of the Lyapunov function along the trajectory of the moment system in  \eqref{eq:moment_system} reveals the dynamics of the Lyapunov function
\begin{align*}
  \dot{V} = \nabla L(e)\cdot \dot{e} = \nabla L(e)\cdot  \Big[\L_*f(m)+\sum_{i=1}^l u_i\L_*g_{i}(m)\Big],
\end{align*}
which can be controlled by $u_1$, $\dots$, $u_l$, the control inputs applied to the ensemble and moment systems, and where `$\cdot$' denotes the Euclidean inner product. In particular, we pick the control inputs $u_1$, $\dots$, $u_l$ such that $\dot{V}(t)<0$ holds for $t>0$ to guarantee $e(t)\rightarrow0$ asymptotically, and the existence of such control inputs is guaranteed by controllabilility of the moment system. To be more specific in how to pick such control inputs, we denote $\bar f(m)=\L_*f(m)$, $\bar g(m)=[\ \L_*g_1(m)\mid\cdots\mid\L_*g_l(m)\ ]$ and $u=[\ \bar u_1\mid\cdots\mid\bar u_l\ ]'$,  then $u$ satisfies the inequality
\begin{align}
\label{eq:feedback_control}
    \nabla L(e) \cdot \big(\bar g(m)u\big) < - \nabla L(e)\cdot \bar f(m),
\end{align}
where $'$ denotes the transpose of matrices. Solving this inequality for $u$ in terms of $m$ will result in a feedback control law that stabilizes the moment system in \eqref{eq:moment_system} to $m_F$ asymptotically. Note that the inequality in \eqref{eq:feedback_control} is linear in $u$, hence it is solvable if $\nabla L(e)\bar g(m) \neq0$, the $l$-dimensional zero row vector, along the moment trajectory. Equivalently, it suffices to argue that the moment system can be controlled to avoid intersecting the solution set $S$ of the equation $\nabla L(e)\bar g(m) =0$ in $\M_p^n$. From the knowledge of differential geometry, we know that $S$ is a codimension $l$ submanifold of $\M_p^n$. Moreover, we can always pick a Lyapunov function $L$ so that the function $\nabla L(e)\bar g(m)$ is proper, e.g., $L$ is smooth and compactly supported. In this case, $S$ is a compact submanifold of $\M_p^n$, and hence $S\backslash\M_p^n$ is path connected so that the moment system can be driven away from $S$, guaranteed by its controllability.

The above discussion gives a constructive proof for the following theorem.
\begin{thm}
\label{thm:moment-feedback}
For any ensemble controllable control-affine  system on $L^p(\O,\mathbb{R}^n)$ in the form of \eqref{eq:ensemble_affine}, there exists a moment-feedback control law $u(m(t))=(u_1(m(t)),\dots,u_l(m(t)))$ asymptotically stabilizing the system to any final profile in $L^p(\O,\mathbb{R}^n)$, where $m(t)$ denotes the moment sequence of the state of the system. 
\end{thm}
\begin{rem}
Note that although Theorem \ref{thm:moment-feedback} is only stated for the case of labeled ensemble systems with classical ensemble moments, it definitely works for unlabeled ensemble systems with output ensemble moments as well. Specifically, in such a case, the considered unlabeled ensemble system can be asymptotically stabilized to any desired pattern by using an output moment-feedback control, given that the system is pattern controllable.
\end{rem}
In the rest of the paper, we revisit Examples \ref{ex:Bloch} and \ref{ex:Nonlinear_ensemble} to accomplish some practical ensemble control tasks for the systems in \eqref{eq:bloch} and \eqref{eq:nonlinear_ex} by using the moment-feedback control design method.

\begin{ex}
A typical control task in NMR spectroscopy is to steer the ensemble Bloch system in \eqref{eq:bloch} from $x(0,\beta)=(0,0,1)'$ to $x(T,\beta)=(1,0,0)'$ for all $\beta\in[1-\delta,1+\delta]$, which corresponds to a uniform $\pi/2$ rotation about $y$-axis. The control input that achieves this task can be obtained by equivalently steering the associated moment system from $m_0$ to $m_F$, where $m_0$ and $m_F$ are the moment sequences corresponding to the constant functions $x(0,\beta)$ and $x(T,\beta)$, respectively. In particular, we consider a 10\% variation in the strength of the applied rf field, i.e., $\delta=0.1$ so that $\beta \in[0.9,1.1]$.

To synthesize a moment-feedback control law, we truncate the moment sequence of the Bloch ensemble up to order $N$, and then define a candidate Lyapunov function $$V(t)=L(e(t))=\sum_{j=1}^Ne_j(t)'e_j(t),$$ 
where $e_j(t)=m_j(t)-m_F\in\mathbb{R}^3$. In Example \ref{ex:Bloch}, we have derived the moment dynamics of the Bloch ensemble, then following the control design procedure shown above Theorem \ref{thm:moment-feedback} yields
\begin{align}
    \begin{cases}
    u(t)=-\sum_{j=1}^N(e_{1,j}(t)m_{3,j+1}(t)-e_{3,j}(t)m_{1,j+1}(t))\\
    v(t)=0
    \end{cases}
    \label{eq:control_Bloch}
\end{align}
where $e_{i,j}$ and $m_{i,j}$ denote the $i^{\rm th}$ components of $e_j$ and $m_j$, respectively. 

In the simulation, we choose $N=35$ and apply the designed control inputs in \eqref{eq:control_Bloch} to both the Bloch ensemble and its moment system. The simulation results are shown in Figures \ref{fig:bloch_moments} and \ref{fig:bloch_states}. In particular, Figure \ref{fig:bloch_moments} shows the moment trajectory $m(t)$ and error trajectory $e(t)$ obtained by applying the designed control inputs to the moment system in \eqref{eq:bloch_moment_dyn}, and Figure \ref{fig:bloch_states} shows the state trajectories of 300 individual systems in the Bloch ensemble and the control inputs as functions of the time $t$. Moreover, we observe from these two figures that the Bloch ensemble and its moment system are simultaneously steered to neighborhoods of $x(T,\beta)$ and $m_F$, respectively, which in turn validates the proposed moment-feedback control design method.



\begin{figure}
    \centering
    \includegraphics[width=0.85\columnwidth,keepaspectratio]{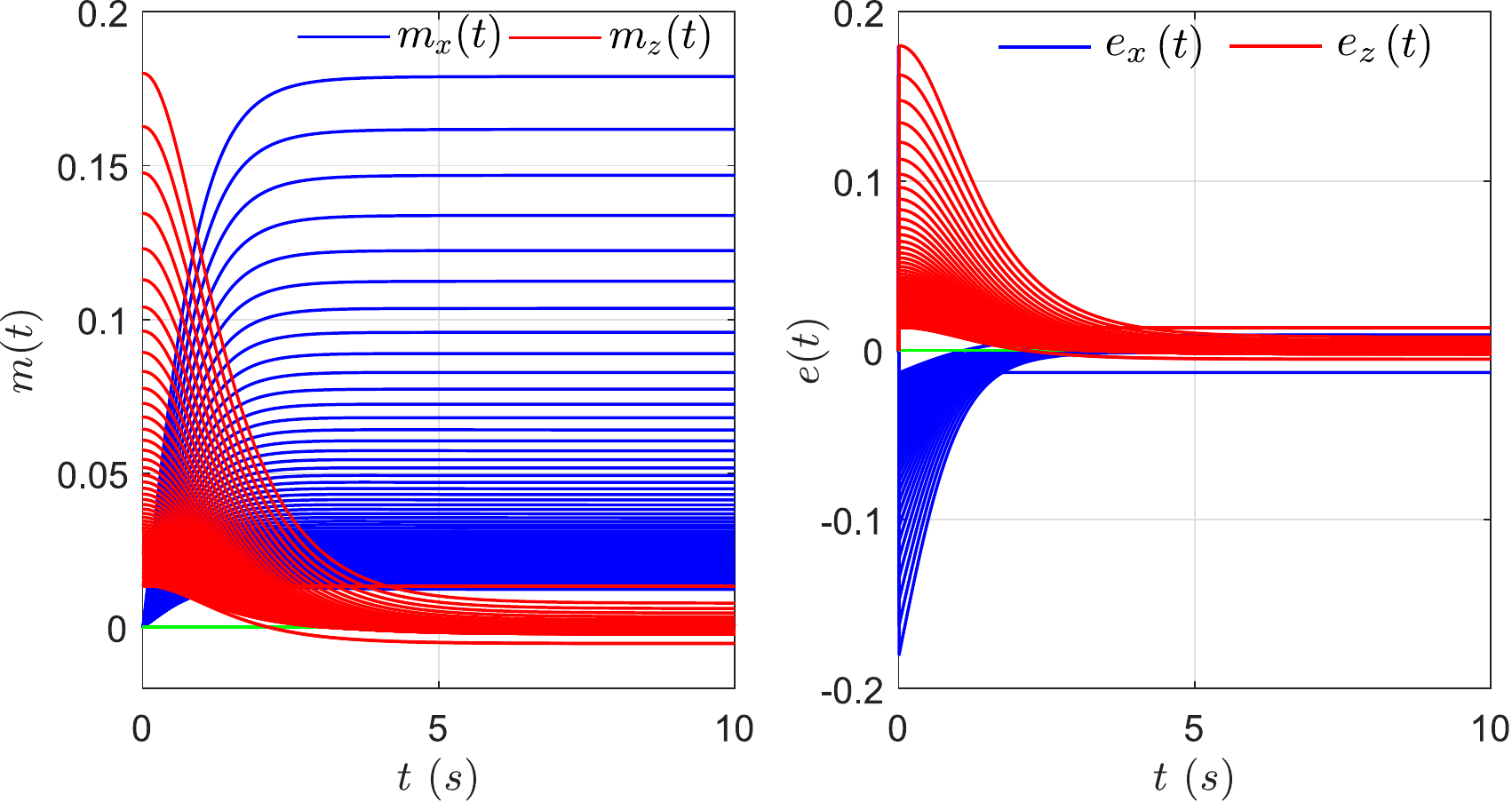}
    \caption{\footnotesize \noindent (Left) States of the controlled moment system of $(N= 35)$. The moment sequences corresponding to $x_1,x_3$ are denoted as $m_x,m_z$, respectively. (Right) Error trajectories of moment sequences, i.e., $m(t)-m_F$. The error trajectories corresponding to $m_x,m_z$ are denoted as $e_x,e_z$, respectively. Control law was derived using the Lyapunov approach, resulting in $\dot{V}(t) \le 0$. Thus, stabilizing $e(t)$ in a neighborhood of 0.}
    \label{fig:bloch_moments}
\end{figure}
\begin{figure}
    \centering
    \includegraphics[width=0.97\columnwidth,keepaspectratio]{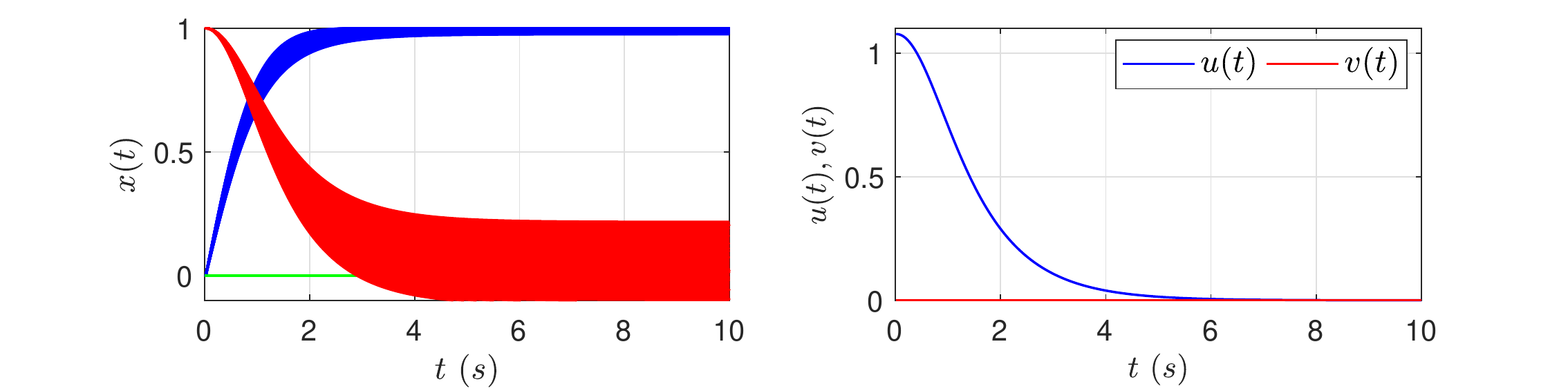}
    \caption{\footnotesize \noindent (Left) States of the controlled ensemble Bloch system. (Right) The feedback control trajectories designed using the moment stabilization approach. The parameter space was discretized to get 300 Bloch systems.}
    \label{fig:bloch_states}
\end{figure}
\end{ex}

\begin{ex}
In this example, we revisit the ensemble system in \eqref{eq:nonlinear_ex} in Example \ref{ex:Nonlinear_ensemble} with the control task to steer the system from $x_0(\beta)=(2,1)$ to $x_F(\beta)=(1,0)$, whose moment sequences are denoted by $m_0$ and $m_F$, respectively, uniformly for all $\beta\in[0.5,1]$. 

Similar to the previous example, we still use the Lyapunov function $V(t)=L(e(t))=\frac{1}{2} \sum_{j=1}^N e_j'(t)e_j(t)$ with $N=50$, and determine a feasible control input following the same moment-feedback control design procedure.
The resulting control input, $u(t) = -\sum_{j=1}^N (5e_{1,j}(t)+e_{2,j}(t))$, was then applied to both of the ensemble and moment systems. The phase portraits of 500 systems in the ensemble and the error trajectory of the moment system are shown in Figure \ref{fig:inv_pend_moment_errors}. In addition, the control input as a function of $t$ and the resulting final profile are documented in Figure \ref{fig:inv_pend_phase}, which also demonstrates that the final profile obtained by applying the designed control input to the system accurately approximates the desired profile.



\begin{figure}
    \centering
    \includegraphics[width=0.95\columnwidth,keepaspectratio]{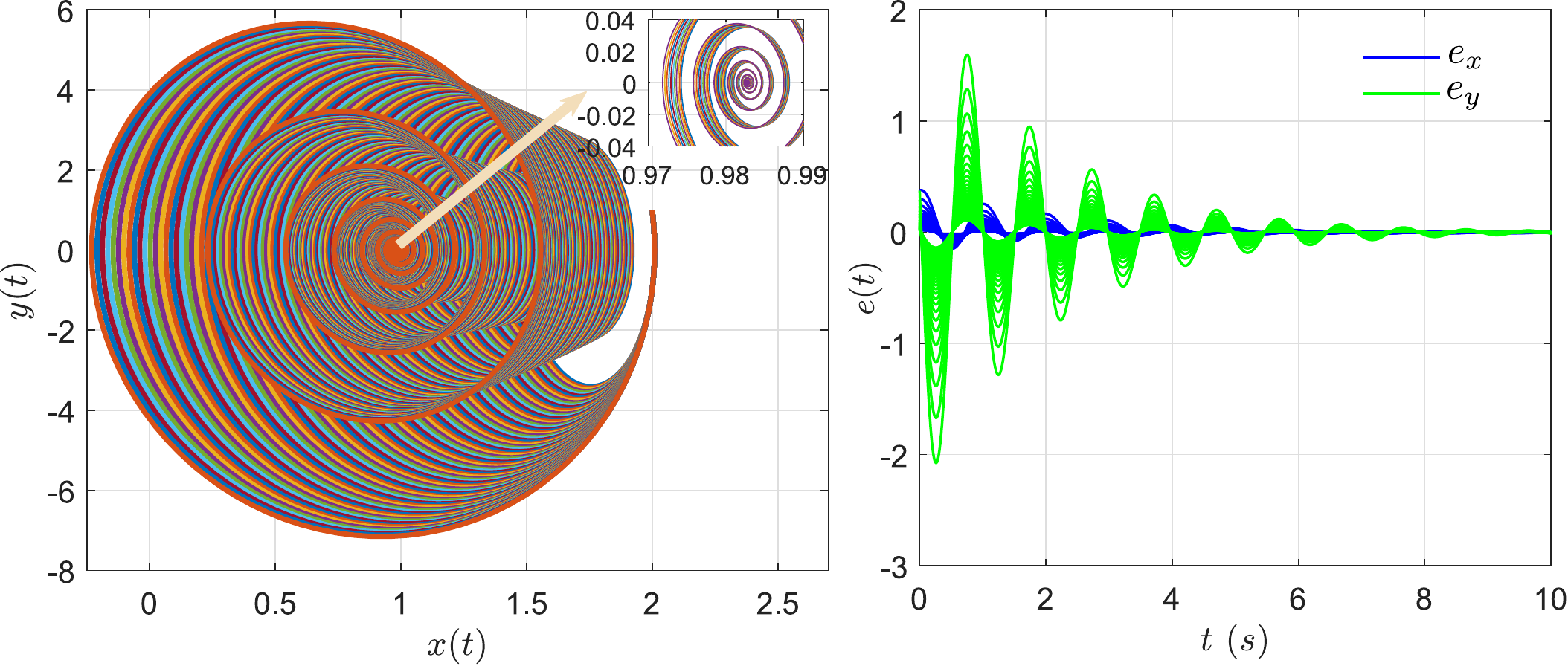}
    \caption{\footnotesize \noindent (Left) Phase trajectories of the controlled ensemble system. The parameter space was uniformly discretized to yield 500 systems. (Right) Error trajectories of the moment sequences, where we considered the moments of order up to 50. The errors corresponding to the moments sequences associated with the states $x,y$ are denoted by $e_x,e_y$, respectively.}
    \label{fig:inv_pend_moment_errors}
\end{figure}
\begin{figure}
    \centering
    \includegraphics[width=0.85\columnwidth,keepaspectratio]{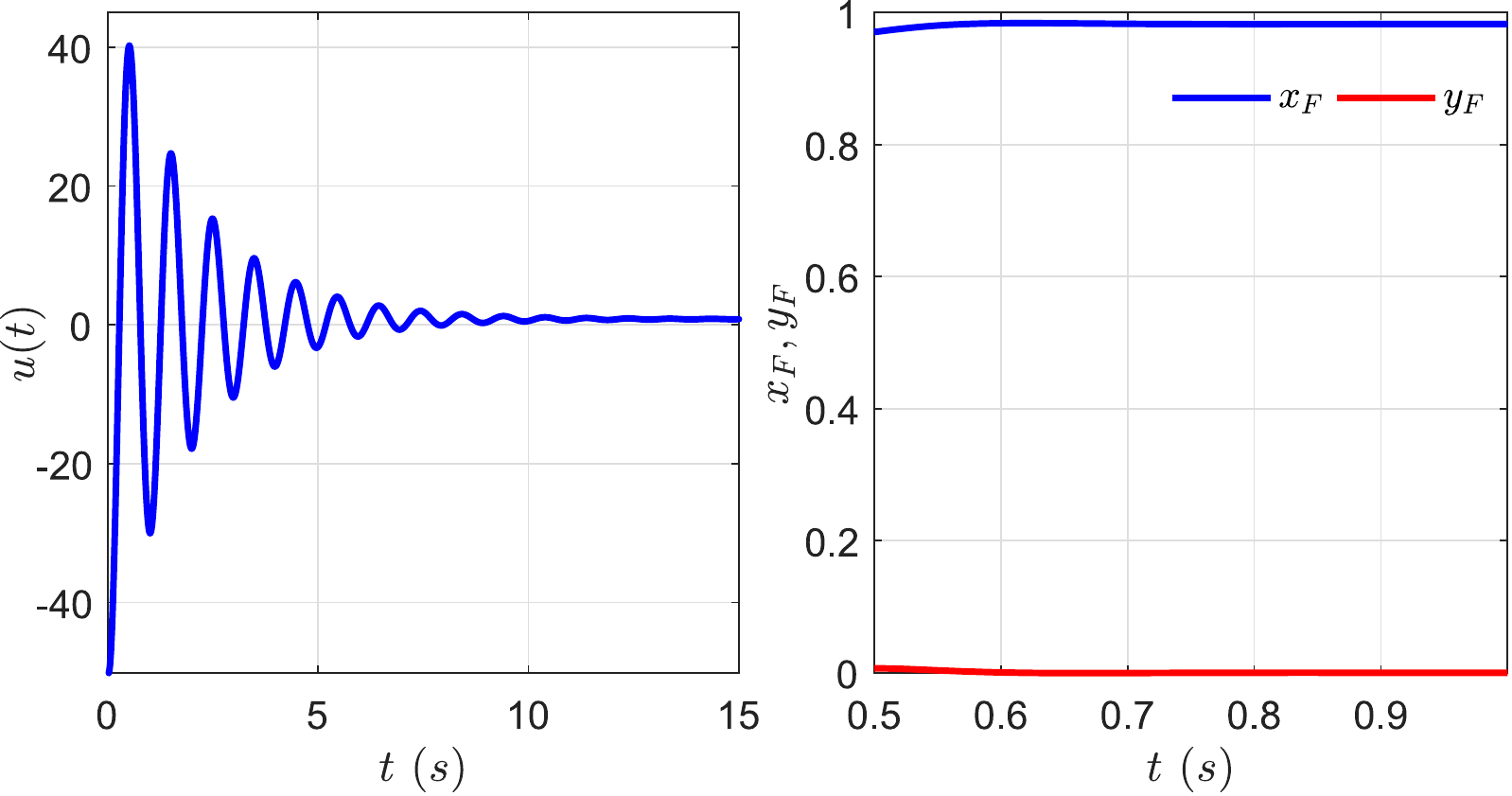}
    \caption{\footnotesize \noindent (Left) Control trajectory for the ensemble system computed as a linear combination of ensemble moments up to order 50. (Right) Final profile of the states of the ensemble system.}
    \label{fig:inv_pend_phase}
\end{figure}
\end{ex}

These numerical and simulation examples illustrate the application of the proposed approach in accommodating aggregated measurements to close the feedback loop in ensemble control systems. 
 \section{Conclusions}\label{sec:conclusions}

In this paper, we develop a moment-based ensemble control framework to close the feedback loop in ensemble control systems using population level, aggregated measurements, which are available in many emerging applications. In particular, we integrate the method of moments in probability and statistics into control theory by interpreting deterministic ensemble systems from the probabilistic viewpoint. The major tool developed in this work is the dynamic Hausdorff moment problem, that is, an extension of the classical Hausdorff moment problem from the aspects of differential geometry and dynamical systems theory, through which we show that ensemble systems and their moment systems share the same control and controllability property. These theoretical investigations lay the foundation for the design of moment-feedback control laws for ensemble systems. Furthermore, this controllability-preserving moment transformation simultaneously accomplishes the task of model reduction of ensemble systems containing uncountable many individual systems to networks containing countably many agents. This nature of the proposed framework further sheds lights on the utilization of the moment-based control for ensemble systems with a focus on computational aspects of control synthesis such as truncation of moments systems, approximation of ensemble and output moments with unbiased sample moments, and their connections to mean-field theory.  

\appendices
\section*{Appendix}
\subsection{Proof of Proposition \ref{prop:Hausdorff_2d}}
\label{appd:Hausdorff_2d}
We first prove the necessity, that is, the moment sequence $m_k$, $k\in\mathbb{N}^2$ of a function $\varphi\in L^2(\O)$ satisfies the inequality in \eqref{eq:Hausdorff_2d}. To simplify the notation, we use the multi-index notations that $\Delta^nm_k=\Delta^{n_1-k_1}_1\Delta^{n_2-k_1}_1m_{k_1k_2}$ and $\b^k(1-\b)^{n-k}=\b_1^{k_1}\beta_2^{k_2}(1-\b_1)^{n-k_1}(1-\b_2)^{n-k_2}$. Comparing the expression of $\Delta^{n-k}m_k$ with the power series expansion of $\b^k(1-\b)^{n-k}$, we obtain an integral representation of $\Delta^{n-k}m_k$ as $\Delta^{n-k}m_k=\int_\O\b^k(1-\b)^{n-k}\varphi(\b)d\b$.
Then, the H\"{o}lder's inequality yields
\begin{align*}
|\Delta^{n-k}m_k|^2&\leq \int_\O\b^k(1-\b)^{n-k}d\b\int_\O\b^k(1-\b)^{n-k}\varphi^2(\b)d\b\\
&=\frac{\int_\O\b^k(1-\b)^{n-k}\varphi^2(\b)d\b}{(n+1){n \choose k}}.
\end{align*}
Therefore, we obtain $(n+1)\sum_{k=0}^n\Big[{n \choose k}\Delta^{n-k}m_k\Big]^2\leq\int_\O\Big[\sum_{k=0}^n{n \choose k}\b^k(1-\b)^{n-k}\Big]\varphi^2(\b)d\b=\int_\O\varphi^2(\b)d\b=\|\varphi\|^2_2$,
where $\sum_{k=0}^n{n \choose k}\b^k(1-\b)^{n-k}
=\sum_{k_1=0}^{n_1}{n_1 \choose k_1}\b_1^{k_1}(1-\b_1)^{n_1-k_1}\sum_{k_2=0}^{n_2}{n_2 \choose k_2}\b_2^{k_2}(1-\b_2)^{n_2-k_2}
=(\b_1+1-\b_1)^{n_1}\cdot(\b_2+1-\b_2)^{n_2}=1$
by the binomial theorem.

Conversely, to show the sufficiency, we assume the double sequence $m_k$, $k\in\mathbb{N}^2$ satisfies the inequality in \eqref{eq:Hausdorff_2d}. Then, the sequence $m_k$ also satisfies the inequality in Lemma \ref{lem:Hausdorff_2d} because taking $n=(n_0,n_0)$ yields $\sum_{k_1,k_2=0}^{n_0}{n_0\choose k_1}{n_0\choose k_2}|\Delta^{n-k}m_k|
\leq (n_0+1)\sqrt{\sum_{k_1,k_2=0}^{n_0}\Big[{n_0\choose k_1}{n_0\choose k_2}\Delta^{n-k}m_k\Big]^2}\leq\sqrt{C}$,
which implies the existence of a signed Borel measure $\mu$ on $\O$ such that $m_k=\int_\O\b^kd\mu(\b)$. Next, pick any $f\in L^2(\O)$, because continuous functions are dense in $L^2$-functions, there is a sequence of continuous functions $f_m$ such that $f_m\rightarrow f$ in $L^2(\O)$. Let $B_n(f_m)=\sum_{k=0}^nf_m\big(\frac{k}{n}\big){n\choose k}\b^k(1-\b)^{n-k}=\sum_{(k_1,k_2)=(0,0)}^{(n_1,n_2)}f_m\big(\frac{k_1}{n_1},\frac{k_2}{n_2}\big){n_1\choose k_1}{n_2\choose k_2}\b_1^{k_1}(1-\b_1)^{n_1-k_1}\b_2^{k_2}(1-\b_2)^{n_2-k_2}$ be the $(n_1,n_2)^{\rm th}$ Bernstein polynomial induced by $f_m$, then $B_n(f_m)\rightarrow f_m$ uniformly as $n_1,n_2\rightarrow\infty$ for each $m$. Note that because $\mu$ is the representing measure of $m_k$, we also have
\begin{align*}
\int_\O B_n(f_m)d\mu&=\sum_{k=0}^nf_m\Big(\frac{k}{n}\Big){n\choose k}\Delta^{n-k}m_k\\
&\leq\sqrt{\sum_{k=0}^n\frac{f_m^2\big(\frac{k}{n}\big)}{n+1}}\sqrt{\sum_{k=0}^n\big[{n\choose k}\Delta^{n-k}m_k\big]^2}\\
&\leq\sqrt{C}\sqrt{\sum_{k=0}^n\frac{f_m^2\big(\frac{k}{n}\big)}{n+1}}.
\end{align*}
Taking the limit as $n_1,n_2\rightarrow\infty$ from both sides and changing the order of the integral and limit on the left-hand side by the uniform convergence, we have $\int_\O f_m d\mu=\sqrt{C\int_\O f_m^2(\b)d\b}$
by the definition of Riemann integrals. Again, taking limit as $m\rightarrow\infty$ and changing the order of integrals and limits from both sides by the dominant convergence theorem, it follows $\int_\O f_m d\mu\leq\sqrt{C\int_\O f(\b)d\b}=\sqrt{C}\|f\|_2$,
which implies the map $f\mapsto\int_\O fd\mu$ defines a bounded linear functional on $L^2(\O)$, i.e., an element in the dual space $L^2(\O)^*$. Because $L^2(\O)=L^2(\O)^*$, we conclude $\int_\O fd\mu=\int_\O f(\b)\varphi(\b)d\b$ for some $\varphi\in L^2(\O)$. Since $f\in L^{2}(\O)$ is arbitrary, $d\mu(\b)=\varphi(\b)d\b$ holds, which also concludes the proof of the sufficiency. The uniqueness of $\varphi$ directly follows from the density of polynomials in $L^2(\O)$.

\bibliographystyle{IEEEtran}
\bibliography{References}
\end{document}